\documentclass[11pt,a4paper,english]{article}
\usepackage[T1]{fontenc}
\usepackage[latin9]{inputenc}
\usepackage{enumitem}
\usepackage{amsmath}
\usepackage{amsthm}
\usepackage{amssymb}

\makeatletter

\theoremstyle{plain}
\newtheorem{thm}{\protect\theoremname}[section]
\theoremstyle{plain}
\newtheorem{lem}[thm]{\protect\lemmaname}
\ifx\proof\undefined
\newenvironment{proof}[1][\protect\proofname]{\par
	\normalfont\topsep6\p@\@plus6\p@\relax
	\trivlist
	\itemindent\parindent
	\item[\hskip\labelsep\scshape #1]\ignorespaces
}{%
	\endtrivlist\@endpefalse
}
\providecommand{\proofname}{Proof}
\fi
\theoremstyle{plain}
\newtheorem{cor}[thm]{\protect\corollaryname}

\usepackage{authblk}
\date{}

\makeatother

\usepackage{babel}
\providecommand{\corollaryname}{Corollary}
\providecommand{\lemmaname}{Lemma}
\providecommand{\theoremname}{Theorem}

\begin{document}
\title{Some remarks about self-products and entropy}
\author{Michael Hochman\thanks{Research supported by ISF research grant 3056/21.}}
\maketitle
\begin{abstract}
Let $(X,\mathcal{B},\mu,T)$ be a probability-preserving system with
$X$ compact and $T$ a homeomorphism. We show that if every point
in $X\times X$ is two-sided recurrent, then $h_{\mu}(T)=0$, resolving
a problem of Benjamin Weiss, and that if $h_{\mu}(T)=\infty$ then
every full-measure set in $X$ contains mean-asymptotic pairs (i.e.
the associated process is not tight), resolving a problem of Ornstein
and Weiss. 
\end{abstract}

\section{\label{sec:Introduction}Introduction}

Many dynamical properties of a continuous or measure-preserving transformation
$T$ can be discerned from the behavior of the self-product map $T\times T$.
For example, in the category of probability preserving transformations,
continuous spectrum is equivalent to ergodicity of the self-product
(this is weak mixing), and in the topological category of homeomorphisms
of compact metric spaces, distality is equivalent to the self-product
decomposing into disjoint minimal systems. This note concerns two
relatively recent additions to this list, which connect the behavior
of self-products to entropy:
\begin{itemize}
\item A topological system $(X,T)$ is called \textbf{doubly minimal }if
for every $x,y\in X$ that do not lie on the same orbit, $(x,y)$
has dense two-sided orbit under $T\times T$. Benjamin Weiss \cite{Weiss1995}
showed that this property implies $h_{top}(T)=0$, and conversely,
if an ergodic measure-preserving system has zero entropy then it can
be realized as an invariant measure on a doubly minimal system.
\item A topological system $(X,T)$ is \textbf{mean distal }if for every
$x\neq y$ in $X$, the Hamming distance
\[
\overline{d}(x,y)=\limsup_{n\rightarrow\infty}\frac{1}{2n+1}\sum_{i=-n}^{n}d(T^{i}x,T^{i}y)
\]
is positive, and an invariant measure $\mu$ on $X$ is called \textbf{tight
}if, after removing a nullset, $\overline{d}(x,y)>0$ for every $x,y$.
Thus, every invariant measure on a mean distal system is tight. In
\cite{OrnsteinWeiss2004}, Ornstein and Weiss showed that positive,
finite entropy precludes tightness, whereas every zero-entropy measure-preserving
system has an extension that is realized on a mean distal system.
\end{itemize}
Our purpose in this note is to fill in two missing pieces of this
story.

In his study of doubly minimal systems, Weiss observed that double
minimality implies that every point in $X\times X$ is two-sided recurrent,
a property that we shall call \textbf{double recurrence}, and asked
whether double recurrence by itself implies entropy zero. 

If one asks for forward (instead of two-sided) recurrence of every
pair, then the answer is affirmative, because every positive entropy
system possesses an off-diagonal positively-asymptotic pair \cite{BlanchardHostRuette2002}.
From this it follows that a positive-entropy system $X$ cannot have
all points recurrent in $X^{4}$, since one can take $(x_{1},x_{2})$
forward asymptotic and $(x_{3},x_{4})$ backward asymptotic, with
$x_{1}\neq x_{2}$ and $x_{3}\neq x_{4}$ and then $(x_{1},x_{2},x_{3},x_{4})\in X^{4}$
is not recurrent. In unpublished work Weiss extended this conclusion
to $X^{3}$, but the original question for $X\times X$ was not resolved.
This is our first result:
\begin{thm}
\label{thm:double-recurrence}Let $(X,\mathcal{B},\mu,T)$ be an invertible
ergodic measure preserving system on a compact metric space. If $h_{\mu}(T)>0$,
then there exists $(x,x')\in X\times X$ that is not two-sided recurrent
under $T\times T$.
\end{thm}
In view of this and of Weiss's realization result on doubly minimal
systems, it follows that a measure preserving system has positive
entropy if and only if it can be realized on a doubly recurrent system.

Our second result concerns tightness. Ornstein and Weiss showed that
if $0<h_{\mu}(T)<\infty$ then $T$ is not tight. Their proof, however,
did not apply when $h_{\mu}(T)=\infty$, and this case was left open.
Slightly modifying their argument, we show
\begin{thm}
\label{thm:tightness}Infinite entropy systems are never tight.
\end{thm}
Together with the results in \cite{OrnsteinWeiss2004}, this provides
yet another characterization of zero/positive entorpy: A system has
entropy zero if and only if it is a factor of a mean distal system.

The remainder of the paper consists of two sections, one for each
theorem. By convention, all our measure spaces are Borel spaces endowed
with a compatible compact metric, and all sets and measures are Borel.
The transformations $T,S$ defining our dynamical are always homeomorphisms.
Intervals $[a,b]$ are often identified with their integer counterparts,
$[a,b]\cap\mathbb{Z}$. If $(\xi_{i})$ is a sequence then we write
$\xi_{k}^{\ell}=\xi_{k}\xi_{k+1}\ldots\xi_{\ell}$ and sometimes $\xi_{[k,\ell]}=\xi_{k}\xi_{k+1}\ldots\xi_{\ell}$.

\section{\label{sec:Double-recurrence}Double recurrence}

In this section we prove Theorem \ref{thm:double-recurrence}, first
for symbolic systems, and then in general. 

The heuristic of the proof is that a positive-entropy system $(X,T)$
should decompose, approximately, as a product; and if there were an
exact product structure $X=X_{1}\times X_{2}$ for systems $X_{1},X_{2}$
of positive entropy, then we could find a forward asymptotic pair
$(u_{1},u_{2})\in X_{1}\times X_{1}$, and a backward asymptotic pair
$(v_{1},v_{2})\in X_{2}\times X_{2}$. Then the point $((u_{1},v_{1}),(u_{2},v_{2}))\in X\times X$
would not be recurrent. 

Such an exact product structure is available in the measure-theoretic
framework, by the Weak Pinsker property \cite{Austin2018}, but not
in the topological one. Nevertheless, positive entropy still gives
some semblance of independence, in that we can partition time into
complementary periodic sets $E_{1},E_{2}\subseteq\mathbb{Z}$, such
that the behavior of orbits on the two sets is approximately independent.
Using this, we will be able to find $u,v\in X$ which are (``forward'')
asymptotic along $E_{1}^{+}=E_{1}\cap\mathbb{N}$ and (``backward'')
asymptotic along $E_{2}^{-}=E_{2}\cap(-\mathbb{N})$, and $u,v$ differ
on a set all of whose traslates intersect both $E_{1}$ and $E_{2}$.
This will be enough to establish non-recurrence of $(u,v)$.

\subsection{\label{subsec:A-combinatorial-lemma}A combinatorial lemma}

Let $A$ be a finite alphabet and $L\subseteq A^{m}$. We say that
$L$ \textbf{admits a full binary tree }if there are sets $L_{i}\subseteq A^{i}$
for $0\leq i\leq m$, such that $L_{0}$ consists of the empty word,
$L_{m}=L$, and each $a\in L_{i}$ extends in exactly two ways to
$a',a''\in L_{i+1}$. Functions or random variables $W_{1},\ldots,W_{m}:\Omega\rightarrow A$
admit a full binary tree if their image does.
\begin{lem}
\label{lem:binary-trees}For every $0<\eta<1$ and $m\in\mathbb{N}$
there is an $\varepsilon>0$ so that the following holds.

Let $B$ be a finite set and $B_{1},\ldots,B_{m}\subseteq B$. Let
$W=(W_{1},\ldots,W_{m})$ be $B$-valued random variables, and $\mathcal{F}$
a $\sigma$-algebra in the underlying sample space. Assume 
\begin{enumerate}
\item $|B_{i}|>|B|^{\eta}$ .
\item $\mathbb{P}(W_{i}\in B_{i})>1-\varepsilon$ .
\item $\left|H(W|\mathcal{F})-\sum_{i=1}^{m}\log|B_{i}|\right|<\varepsilon\log|B|$.
\end{enumerate}
Then with probability $1-\eta$ over atoms $F\in\mathcal{F}$, the
restriction of $W_{1},\ldots,W_{m}$ to $F$ admits a full binary
tree.
\end{lem}
If we assume in addition that $|B|$ is bounded, then the conditional
distribution of $W_{1}^{m}$ on the atoms of $\mathcal{F}$ will tend
to independent as $\varepsilon\rightarrow0$, and the conclusion is
trivial. But in general there is no such implication. The example
to have in mind is when $(X_{n})_{n=-\infty}^{\infty}$ is an ergodic
process of positive entropy and $W_{i}$ are consecutive blocks of
$n$ variables, $W_{i}=X_{in}\ldots X_{(i+1)n-1}$. If $\mathcal{F}$
is trivial, or a factor algebra that does not exhaust the entropy,
then for small $\eta>0$ and every $m\in\mathbb{N}$ and $\varepsilon>0$,
the random variables $W_{1},\ldots,W_{m}$ will satisfy the hypotheses
of the lemma if $n$ is large enough.
\begin{proof}

Let $E$ denote the event $W_{i}\in B_{i}$ for all $i$. We claim
that we can assume that $\mathbb{P}(E)=1$. By (1), the entropy of
each $W_{i}$ on $E^{c}$ is bounded by $\log|B|<(1/\eta)\log|B_{i}|$,
and by (2), $\mu(E^{c})\leq m\varepsilon$. Thus, restricting $W$
and $\mathcal{F}$ to $E$ changes $H(W|\mathcal{F})$ by no more
than $\varepsilon'\sum\log|B_{i}|$ for an $\varepsilon'$ that can
be made arbitrarily small by decreasing $\varepsilon$, so after restricting,
(3) still holds for a slightly larger $\varepsilon$. 

Now proceed by induction on $m$. When $m=1$, we must show that,
for a $1-\eta$ fraction of $F\in\mathcal{F}$, the variable $W_{1}$
is not a.s. constant on $F$, or equivalently that the conditional
entropy of $W_{1}$ on $F$ is positive. Since $W_{1}\in B_{1}$ we
have a pointwise upper bound $\log|B_{1}|$ on the entropy of $W_{1}$
on each atom of $\mathcal{F}$, and these conditional entropies average
to $H(W_{1}|\mathcal{F})$, which by (3) and (1) is at least $\log|B_{1}|-\varepsilon\log|B|>(1-\varepsilon/\eta)\log|B_{1}|$.
As $\varepsilon$ decreases, these upper (pointwise) and lower (average)
bounds approach each other, so, for small enough $\varepsilon$, on
a fraction of atoms approaching full measure, we get a lower pointwise
bound of the same magnitude. 

For $m>1$, our assmption $W_{i}\in B_{i}$ implies the trivial bounds
$H(W_{1}|\mathcal{F})\leq\log|B_{1}|$ and $H(W_{2}^{m}|\mathcal{F}\lor\sigma(W_{1}))\leq\sum_{i=2}^{m}\log|B_{i}|$.
Using the chain rule for entropy, (3) becomes 
\[
\left|\left(\log|B_{1}|-H(W_{1}|\mathcal{F})\right)+\left(\sum_{i=2}^{m}\log|B_{i}|-H(W_{2}^{m}|\mathcal{F}\lor\sigma(W_{1}))\right)\right|<\varepsilon\log|B|
\]
Since both summands on the left are non-negative, the bound $\varepsilon\log|B|$
applies to both, and this is condition (3) for the sequence $W_{2}^{m}$
and algebra $\mathcal{F}\lor\sigma(W_{1})$. Thus, for $\varepsilon$
small, we can apply the induction hypothesis to $W_{2}^{m}$ and $\mathcal{F}\lor\sigma(W_{1})$
with parameters $\eta^{2}/4$ instead of $\eta$ (the decreasing $\eta$
does not invalidate (1)).

We conclude that on $1-\eta^{2}/4$ of the atoms of $\mathcal{F}\lor\sigma(W_{1})$,
the restriction of $W_{2}^{m}$ admit a full binary tree. Thus, on
$1-\eta/2$ of the atoms $F\in\mathcal{F}$, at least $1-\eta/2$
of the atoms $G\in\sigma(W_{1})$ (with respect to the conditional
measure on $F$) are such that $W_{2}^{m}$ admits a full binary tree
on $F\cap G$. Also, arguing as in the case $m=1$, on a $1-\eta/2$
fraction of atoms $F\in\mathcal{F}$, the conditional distribution
of $W_{1}$ has large entropy, and in particular does not take any
single value with probability higher than $1-\eta/2$. It follows
that, with probability $1-\eta$ over $F\in\mathcal{F}$, there are
two atoms $G_{1},G_{2}\in\sigma(W_{1})$ such that $W_{2}^{m}$ admit
a full binary tree on $F\cap G_{1}$, $F\cap G_{2}$, and therefore
$W_{1}^{m}$ admits a full binary tree on $F$.
\end{proof}
\begin{cor}
\label{cor:separated-point-from-high-entropy-sequences}If $L\subseteq A^{m}$
admits a binary tree then there are $u,v\in L$ such that $u_{i}\neq v_{i}$
for $i=1,\ldots,m$. In particular, in the setting of Lemma \ref{lem:binary-trees},
there exist realizations $u,v$ of $W_{1}^{m}$ satisfying the above
and coming from the same atom of $\mathcal{F}$.
\end{cor}
\begin{proof}
If $L\subseteq A^{m}$ admits a full binary tree, let $L_{i}\subseteq A^{i}$
be as in the definition. Choose any $u\in L_{m}$ and construct $v\in L_{m}$
inductively: start with the empty word and, having constructed a word
$v^{(i)}\in L_{i}$, let $v^{(i+1)}\in L_{i+1}$ be an extension whose
last symbol is different from $u_{i}$; one exists since there are
two ways to distinct extensions in $L_{i+1}$. Set $v=v^{(m)}.$
\end{proof}

\subsection{\label{subsec:symbolic-case}Proof of the theorem in the symbolic
case}

In this section we prove that if $\Omega\subseteq A^{\mathbb{Z}}$
is a subshift of positive entropy then it is not doubly recurrent.
By the variational principle, we may fix an ergodic shift invariant
measure on $\Omega$. Taking $X_{n}:\Omega\rightarrow A$ to be the
coordinate projections, $(X_{n})_{n=-\infty}^{\infty}$ becomes an
$A$-valued stationary ergodic process of positive entropy, and our
goal is to find two realizations $u,v$ of $(X_{n})$ such that $(u,v)$
is not recurrent in $A^{\mathbb{Z}}\times A^{\mathbb{Z}}$.

For $n\in\mathbb{N}$, define an $n$-interval to be a set of the
form $I=[kn,(k+1)n)$ with $k\in\mathbb{Z}$, and say that $I$ is
odd or even according to the parity of $k$. Suppressing the parameter
$n$ from the notation, let 
\[
Y_{i}=(X_{2in},\ldots,X_{(2i+1)n-1})\qquad,\qquad Z_{i}=(X_{(2i+1)n},\ldots,X_{(2i+1)n-1})
\]
be the blocks of variables in even and in odd $n$-intervals, respectively.
Then $\mathbb{Z}$ decomposes into disjoint $n$-intervals which are
alternately odd and even, and similarly 
\[
X_{-\infty}^{\infty}=(\ldots,Y_{-2},Z_{-2},Y_{-1},Z_{-1},Y_{0},Z_{0},Y_{1},Z_{1},\ldots)
\]

Let $h>0$ denote the entropy of the process $(X_{n})$, so that
\[
\frac{1}{n}H(X_{1}^{n})=h+o(1)
\]
(all error terms are asymptotic as $n\rightarrow\infty$). It is not
hard to see that for large $n$, the processes $(Y_{i})$ and $(Z_{i})$
are roughly independent, in the sense that the entropy of each is
$\frac{1}{2}(h+o(1))$, and their joint entropy is $h$. We will not
use this property directly, although it is implicit in the proof below.
Set
\begin{align*}
E^{-} & =\bigcup\{\text{even \ensuremath{n}-intervals to the left of \ensuremath{0}\}}\\
E^{+} & =\bigcup\{\text{odd \ensuremath{n}-intervals to the right of \ensuremath{0}\}}
\end{align*}

\begin{lem}
\label{lem:main-lemma-double-recurrence}$H(X_{-n}^{n-1}|\{X_{i}\}_{i\in E^{-}\cup E^{+}})=H(Z_{-1},Y_{0}|Y_{-\infty}^{-1},Z_{0}^{\infty})=2n(h+o(1))$
.
\end{lem}
In the lemma we have conditioned on both past and future times, and
it is important to note that in general this can lead to a sharp decrease
in entropy, even when the density of the times is small. For example,
if one takes the even blocks in both directions, it can happen that
$H(Y_{0}|Y_{-\infty}^{\infty})=0$ for all $n$, as can be shown using
a construction similar to that in \cite{OrnsteinWeiss1975}. The validity
of the lemma relies crucially on the asymmetry between the blocks
that we condition on in the past and in the future.
\begin{proof}
[Proof of the Lemma.]The inequality $\leq$ is trivial since 
\[
H(X_{-n}^{n-1}|Y_{-\infty}^{-1},Z_{0}^{\infty})\leq H(X_{-n}^{n-1})=2n(h+o(1))
\]
Thus, we need only prove $\geq$. Note that $X_{-n}^{n-1}=(Z_{-1},Y_{0})$
and
\begin{align*}
H(Z_{-1},Y_{0}|Y_{-\infty}^{-1},Z_{0}^{\infty}) & =H(Y_{0}|Y_{-\infty}^{-1},Z_{0}^{\infty})+H(Z_{-1}|Y_{-\infty}^{0},Z_{0}^{\infty})\\
 & \geq H(Y_{0}|Y_{-\infty}^{-1},Z_{-\infty}^{\infty})+H(Z_{-1}|Y_{-\infty}^{\infty},Z_{0}^{\infty})
\end{align*}
It suffices to show that each of the summands on the right hand side
are $n(h+o(1))$. We prove this for the first summand, the second
is similar.

Fix a large integer $\ell$, and note that
\begin{align*}
4\ell n(h+o(1)) & =H(X_{-2\ell n},X_{-\ell n+1},\ldots,X_{2\ell n-1})\\
 & =H(Y_{-\ell},Z_{-\ell},Y_{-\ell+1},Z_{-\ell+1},\ldots,Y_{\ell-1},Z_{\ell-1})\\
 & =H(Z_{-\ell}^{\ell-1})+H(Y_{-\ell}^{\ell-1}|Z_{-\ell}^{\ell-1})\\
 & =H(Z_{-\ell}^{\ell-1})+\sum_{i=-\ell}^{\ell-1}H(Y_{i}|Y_{-\ell}^{i-1},Z_{-\ell}^{\ell-1})\\
 & =2\ell n(h+o(1))+\sum_{i=-\ell}^{\ell-1}H(Y_{0}|Y_{-\ell-i}^{-1},Z_{-\ell-i}^{\ell-i-1})
\end{align*}
Rearranging and dividing by $2\ell$, we have
\[
\frac{1}{2\ell}\sum_{i=-\ell}^{\ell-1}H(Y_{0}|Y_{-\ell-i}^{-1},Z_{-\ell-i}^{\ell-i-1})=n(h+o(1))
\]
On the other hand, by Martingale convergence and monotonicity of entropy
under conditioning,
\[
\lim_{k,m\rightarrow\infty}H(Y_{0}|Y_{-k}^{-1},Z_{-k}^{m})=\inf_{k,m}H(Y_{0}|Y_{-k}^{-1},Z_{-k}^{m})=H(Y_{0}|Y_{-\infty}^{-1},Z_{-\infty}^{\infty})
\]
inserting this in the previous equation and letting $\ell\rightarrow\infty$
gives
\[
H(Y_{0}|Y_{-\infty}^{-1},Z_{-\infty}^{\infty})=n(h+o(1))
\]
as required.
\end{proof}
Returning to the proof of the theorem, split $[-n,n)$ into disjoint
intervals $I_{1},\ldots,I_{8}$ of length $n/4$ (we can assume $n/4\in\mathbb{N}$),
and note that every sub-interval $J\subseteq[-n,n)$ of length $n/2$
contains one of the $I_{i}$. Let 
\[
W_{i}=(X_{p})_{p\in I_{i}}\qquad i=1,\ldots,8
\]
We would like to apply Corollary \ref{cor:separated-point-from-high-entropy-sequences}
to $W_{1},\ldots,W_{8}$ and the $\sigma$-algebra $\mathcal{F}=\sigma(Y_{-\infty}^{-1},Z_{0}^{\infty})$,
with suitable sets $B,B_{1},\ldots,B_{8}$. To set things up, apply
the Shannon-MacMillan theorem to find a set $A'\subseteq A^{n/4}$
of size 
\[
|A'|=2^{(h+o(1))n/4}=|A|^{(c+o(1))n}
\]
for $c=\log2/(4\log|A|)>0$, and satisfying
\[
\sum_{a\in A'}\mu(a)=1-o(1)
\]
Set $\eta=c/2$, $B=A^{n/4}$ and $B_{1},\ldots,B_{8}=A'$ in the
hypothesis of Corollary \ref{cor:separated-point-from-high-entropy-sequences}.
Also, by Lemma \ref{lem:main-lemma-double-recurrence},
\begin{align*}
H(W_{1}^{8}|Y_{-\infty}^{-1},Z_{0}^{\infty})=H(X_{-n}^{n-1}|Y_{-\infty}^{-1},Z_{0}^{\infty}) & =2(h-o(1))n
\end{align*}
but we also have $H(W_{i})=(h-o(1))\cdot n/2$, so
\[
H(W_{1}^{8}|Y_{-\infty}^{-1},Z_{0}^{\infty})\leq\sum_{i=1}^{8}H(W_{i})\leq2(h+o(1))n
\]
Combining these bounds,
\[
\left|H(W_{1}^{8})-\sum_{i=1}^{8}H(W_{i})\right|=o(n)
\]
which, assuming $n$ is large, completes the hypotheses of Corollary
\ref{cor:separated-point-from-high-entropy-sequences}. 

The corollary now provides us with samples $u,v$ for the process
satisfying
\begin{itemize}
\item $u,v$ are realized on the same atom of $\mathcal{F}=\sigma(Y_{-\infty}^{-1},Z_{0}^{\infty})$,
hence $u,v$ agree on every $n$-interval in $E^{-}\cup E^{+}$.
\item $u|_{I_{i}}\neq v|_{I_{i}}$ for $i=1,\ldots,8$. That is, for each
$i$, there is $j\in I_{i}$ with $u_{j}\neq v_{j}$. 
\end{itemize}
To conclude the proof we must show that $(u,v)$ is not recurrent.
Indeed, let $(u',v')$ be a shift of $(u,v)$ by $k$, with $|k|>2n$.
\begin{itemize}
\item Since $E^{-},E^{+}$ consist of $n$-intervals separated by gaps of
length $n$, after shifting $E^{-}\cup E^{+}$ by $k$, some $n$-interval
$J\subseteq(E^{-}\cup E^{+})+k$ will intersect $[-n,n)$ in a set
of size at least $n/2$, implying that $u'_{j}=v'_{j}$ for all $j\in J\cap[-n,n)$. 
\item Every interval of length $n/2$ in $[-n,n]$, and in particular the
interval $J\cap[-n,n)$, contains one of the $I_{i}$. Thus, there
is then a $j\in I_{i}\subseteq J$ such that $u_{j}\neq v_{j}$. 
\end{itemize}
It follows that $(u,v)$ and $(u',v')$ differ on at least one coordinate
$j\in[-n,n)$. This holds for all shifts $(u',v')$ of $(u,v)$ by
$2n$ or more, so $(u,v)$ is not recurrent.

\subsection{\label{subsec:The-general-case}The general case}

Let $(X,T)$ be a topological system and $\mu$ a $T$-invariant Borel
probability measure of positive entropy. Our plan is as follows.

First, we pass to a measure-theoretic factor $(X,\mu)\rightarrow(Y,\nu)$,
where $Y$ is a subshift over a countable alphabet (for concreteness,
$Y\subseteq\mathbb{N}^{\mathbb{Z}}$), and the entropy of $Y$ is
positive and finite. This factor will be defined using a partition
$\mathcal{C}=\{C_{1},C_{2},\ldots\}$ of $X$ (modulo $\mu$) into
closed sets that are separated from each other in the sense that,
for every $n$, no pair of atoms $C_{i},C_{j}$ with $i,j<n$ can
be simultaneously $\delta_{n}$-close to a third atom. 

Next, pass to a further factor $(Y,\nu)\rightarrow(Z,\theta)$, chosen
so that $Y$ has a finite generator relative to $Z$, and, conditioned
on $Z$, the relative entropy is positive. This factor is obtained
by merging a large finite number of symbols $1,\ldots,r$ in the alphabet
of $Y$.

We can now run a relative version of the argument from the symbolic
case on $Y$, conditioned on $Z$. This yields a pair of points $u,v\in Y$
that lie above the same point in $Z$, and such that $(u,v)$ is not
recurrent in $Y\times Y$. More precisely, the lack of recurrence
arises because every large enough shift $(u',v')$ of $(u,v)$ admits
an index $j\in[-n,n)$ at which $u',v'$ display a common symbol $k$,
while $u,v$ display distinct symbols from among $1,\ldots,r$.

Finally, taking preimages $x,y\in X$ of $u,v\in Y$, respectively,
we find that any large enough shift $(x',y')$ of $(x,y)$ can be
brought, after another bounded shift, into an atom $C_{k}$, whereas
the corresponding shifts of $x,y$ lie in distinct atoms from among
$C_{1},\ldots,C_{r}$. The separation properties of $\mathcal{C}$
now ensure that $(x',y')$ is at least $\delta_{r}$-far from $(x,y)$,
which is non-recurrence.

We now give this argument in detail.

\subsubsection*{Step 1: Constructing the factor $X\rightarrow Y$}

We inductively construct a sequence of disjoint closed sets $C_{1},C_{2},\ldots\subseteq X$
that exhaust the measure $\mu$.

Begin with a finite measurable partition $\mathcal{A}=\{A_{1},\ldots,A_{n_{1}}\}$
of positive entropy and $\varepsilon>0$, and choose closed sets $C_{i}\subseteq A_{i}$
such that $\mu(\bigcup_{i=1}^{n_{1}}C_{i})>1-\varepsilon_{1}$.

Set $n_{k}=2^{k-1}n_{1}$ and suppose that after $k$ steps we have
defined disjoint closed sets $C_{1},\ldots,C_{n_{k}}$. Let $\varepsilon_{k+1}$
be given. For $i=1,\ldots,n_{k}$ let $A_{i}^{k},\subseteq X$ denote
the measurable sets that form the ``Voronoi anuli'' of the sets
$C_{1},\ldots,C_{n_{k}}$ :
\[
A_{i}^{k}=\left\{ x\in X\left|\begin{array}{c}
d(x,C_{i})<d(x,C_{j})\text{ for }j<i\text{ and }\\
d(x,C_{i})\leq d(x,C_{j})\text{ for }j>i
\end{array}\right.\right\} \setminus C_{i}
\]
The sets $A_{1}^{k},\ldots,A_{n_{k}}^{k}$ are measurable, pairwise
disjoint, and together with $C_{1},\ldots,C_{n_{k}}$ they form a
partition of $X$. Now choose closed sets $C_{n_{k}+i}\subseteq A_{i}^{k}$,
$i=1,\ldots,n_{k}$, satisfying $\mu(\bigcup_{i=1}^{n_{k+1}}C_{i})>1-\varepsilon_{k+1}$.
We have defined $C_{1},\ldots,C_{n_{k+1}}$.

Let $\mathcal{C}=\{C_{1},C_{2},\ldots\}$ denote the resulting family
of sets. Observe that:
\begin{enumerate}
\item If $\varepsilon_{k}\rightarrow0$ then $\mathcal{C}$ is a partition
of $X$ up to $\mu$-null sets.
\item If $\varepsilon_{k}\rightarrow0$ quickly enough, $H_{\mu}(\mathcal{C})<\infty$.
In particular, $h_{\mu}(T,\mathcal{C})<\infty$ .

Proof: The $n_{k}=2^{n}n_{1}$ partition elements added at stage $k$
of the construction contribute at most $\varepsilon_{k}\cdot k\log n_{1}$
to the total entropy, so taking $\varepsilon_{k}=1/k^{3}$, for example,
gives $H_{\mu}(\mathcal{C})\leq\sum_{k}\frac{1}{k^{2}}\log n_{1}<\infty$.
\item If $\varepsilon_{k}\rightarrow0$ quickly enough and $\varepsilon_{1}$
is small enough, then $h_{\mu}(T,\mathcal{C})>0$ .

Proof: For any $\varepsilon_{2},\varepsilon_{3},\ldots$ such that
$H_{\mu}(\mathcal{C})<\infty$, the Rohlin distance $H(\mathcal{A}|\mathcal{C})+H(\mathcal{C}|\mathcal{A})$
tends to zero as $\varepsilon_{1}\rightarrow0$, hence $h_{\mu}(T,\mathcal{C})\rightarrow h_{\mu}(T,\mathcal{A})>0$
\item For every $\ell\neq m$ , 
\[
\delta_{m,\ell}=\inf_{i\in\mathbb{N}}\left\{ \max\{d(C_{\ell},C_{i}),d(C_{m},C_{i})\}\right\} >0
\]
Proof: Suppose $1\leq\ell,m<n_{k}$. Because the sets $C_{1},\ldots,C_{n_{k}}$
are closed and disjoint, the infimum above is positive as $i$ ranges
over $\{1,\ldots,n_{k}\}$. For $i>n_{k}$, each $C_{i}$ is contained
in the Voronoi cell of exactly one of the previously defined sets,
and in particular either $d(C_{i},C_{\ell})\geq\frac{1}{2}d(C_{\ell},C_{m})$
or $d(C_{i},C_{m})\geq\frac{1}{2}d(C_{\ell},C_{m})$. Thus, the infimum
is positive over all $i$.
\end{enumerate}
Assume that $\varepsilon_{k}$ have been chosen so that properties
(1)-(3) above are satisfied, and let $\pi:X\rightarrow\mathbb{N}^{\mathbb{Z}}$
be the associated measurable map, 
\[
\pi(x)_{i}=m\qquad\Longleftrightarrow\qquad T^{i}x\in C_{m}
\]
Let $\nu=\pi\mu$ be the push-forward measure, which is invariant
under the shift $S$ on $\mathbb{N}^{\mathbb{Z}}$. Then 
\[
h_{\nu}(S)=h_{\mu}(T,\mathcal{C})\in(0,\infty)
\]

\subsubsection*{Step 2 : Constructing the factor $Y\rightarrow Z$}

Let $\mathcal{D}$ denote the cylinder partition of $\mathbb{N}^{\mathbb{Z}}$
and let $\mathcal{D}^{(r)}$ denote the partition of $Y$ obtained
by merging the first $r$ symbols into a single atom. Since $H_{\nu}(\mathcal{D})<\infty$,
\[
H_{\nu}(\mathcal{D}^{(r)})\rightarrow0\qquad\text{as }r\rightarrow\infty
\]
and since $h_{\nu}(S)>0$, we can choose $r$ large enough that $H_{\nu}(D^{(r)})<h_{\nu}(S)$,
hence
\[
h_{\nu}(S,\mathcal{D}^{(r)})<h_{\nu}(S)
\]

Consider the factor algebra $\mathcal{E}=\bigvee_{i\in\mathbb{Z}}S^{-i}\mathcal{D}^{(r)}$.
Evidently,
\[
h_{\nu}(S|\mathcal{E})=h_{\nu}(S)-h(S|_{\mathcal{E}})>0
\]
Let $\mathcal{B}$ denote the partition of $\mathbb{N}^{\mathbb{Z}}$
obtained by identifying all symbols $r+1,r+2,\ldots$, and observe
that $\mathcal{D}=\mathcal{B}\lor\mathcal{D}^{(r)}$, so $\mathcal{B}$
generates relative to $\mathcal{E}$, hence
\[
h_{\nu}(S,\mathcal{B}|\mathcal{E})=h_{\nu}(S|\mathcal{E})>0
\]

\subsubsection*{Step 3 : Applying the symbolic argument}

We can now run the entire argument from the finite-alphabet case for
the process $(X_{i})$ determined by $\mathcal{B}$, but conditioning
the whole while on $\mathcal{E}$. We will then find an $n$ and samples
$u,v\in\mathbb{N}^{\mathbb{Z}}$ of the process which lie in the same
atom of $\sigma(Y_{-\infty}^{-1},Z_{1}^{\infty})\lor\mathcal{E}$,
and with the property that if $(u',v')$ is a shift of $(u,v)$ by
more than $n$ in either direction, there exists $-n\leq j\leq n$
such that $u_{j}\neq v_{j}$ and $u'_{j}=v'_{j}$. Importantly, $u_{j}\neq v_{j}$
implies $u_{j},v_{j}\in\{1,\ldots,n\}$, because otherwise they would
not lie in the same atom of $\mathcal{E}$.

We must adjust one minor point in this argument: we should work with
the process defined by the partition $\pi^{-1}\mathcal{B}$ on $(X,\mu,T)$
instead of by $\mathcal{B}$ on $(\mathbb{N}^{\mathbb{Z}},\nu,S)$,
and the algebra $\pi^{-1}\mathcal{E}$. The two processes have the
same distribution, but this change ensures that the resulting samples
$u,v$ lie in the image of $\pi$. We remark, that the fact that the
process is not defined on a subshift is not a problem: the proof in
Section \ref{subsec:symbolic-case} the symbolic case we not rely
on the underlying subshift, only on the process.

\subsubsection*{Step 4 : Lifting $u,v$ to $X$ }

Let $x\in\pi^{-1}(u)$ and $y\in\pi^{-1}(v)$. To conclude the argument
we shall show that $(x,y)$ is non-recurrent for $T\times T$. Indeed,
let
\[
\delta=\min_{1\leq\ell<m\leq n}\delta_{m,\ell}>0
\]
For $i\in\mathbb{Z}\setminus\{0\}$ write $x'=T^{i}x$ and $y'=T^{i}y$,
so $\pi x'=u'$ and $\pi y'=v'$. By assumption there exist s $j\in\{-n,\ldots,n\}$
such that $\ell=u_{j}$, $m=v_{j}$ satisfy $1\leq\ell\neq m\leq n$,
and there is a $k\in\mathbb{N}$ such that $u_{j}=v_{j}=k$. In other
words, $S^{j}x\in C_{\ell}$, $S^{j}y\in C_{m}$ and $S^{j}x',S^{j}y'\in C_{k}$.
Thus, using the max-metric $d_{\infty}$ on $X\times X$,
\begin{align*}
\delta & \leq\delta_{\ell,m}\\
 & \leq\max\{d(C_{\ell},C_{k}),d(C_{k},C_{m})\}\\
 & \leq\max\{d(T^{j}x,T^{j}x'),d(T^{j}y',T^{j}y)\}\\
 & =d_{\infty}(T^{j}(x,y),T^{j}(x',y'))
\end{align*}
where for brevity we wrote $T$ for $T\times T$. Since $-n\leq j\leq n$
and $T$ is a homeomorphism, this implies
\[
d_{\infty}(T^{j}(x,y),T^{j}(x',y'))>\delta'
\]
for some $\delta'>0$ depending only on $\delta,n$. Since $i$ was
arbitrary and $(x',y')=T^{i}(x,y)$ this proves non-recurrence of
$(x,y)$.

\section{\label{sec:Tightness}Infinite entropy systems are not tight}

In this section we discuss Theorem \ref{thm:tightness}, which says
that a measure-preserving system of infinite entropy is not tight.
Our contribution to this topic is a small adaptation of the the argument
appearing in the original paper of Ornstein and Weiss for the finite-entropy
case \cite{OrnsteinWeiss2004}. We do not reproduce the full proof,
but briefly outline the argument and the necessary changes. 

\subsection{The original argument}

In the original proof, one begins with a measure preserving system
$(X,\mathcal{B},\mu,T)$ with $X=\{1,\ldots,a\}^{\mathbb{Z}}$, $T$
the shift, and $h_{\mu}(T)\in(0,\infty)$. Let $\mathcal{P}=\{P_{1},\ldots,P_{a}\}$
be the partition according to the symbol at time zero. Let $E\subseteq X$
be a null set; we wish to find $x,y\in X\setminus E$ with $\overline{d}(x,y)=0$.

Choose an open set $U\supseteq E$ of measure small enough that the
partition
\[
\mathcal{Q}_{0}=\{X,U\setminus P_{1},\ldots,X\setminus P_{a}\}
\]
generates a factor $X_{0}$ of $X$ of entropy less $h_{\mu}(T)$.
Now one constructs a tower of measure-theoretic factors $X_{n}$ between
$X$ and $X_{0}$,
\[
X\rightarrow\ldots\rightarrow X_{2}\rightarrow X_{1}\rightarrow X_{0}
\]
with $X_{n}$ (generated by a finite partition $\mathcal{Q}_{n}$
refining $\mathcal{Q}_{n-1}$), all with entropy strictly less than
$h_{\mu}(X)$, along with numbers $N_{1},N_{2}\ldots\in\mathbb{N}$,
such that
\begin{quote}
If $x,y\in X$ map to the same point in $X_{n}$, then the finite
sequences $x|_{[-N_{n},N_{n}]},y|_{[-N_{n},N_{n}]}$ agree on at least
a $1-\frac{1}{n}$ faction of their coordinates.
\end{quote}
Assuming this, for each $n$, use the fact that $X\rightarrow X_{n}$
has positive relative entropy to choose a pair $x^{(n)}\neq y^{(n)}$
that lie above the same point in $X_{n}$. This means that $x^{(n)},y^{(n)}$
project to the same point in $X_{0}$ as well. By shifting the points
if necessary, we can assume $x_{0}^{(n)}\neq y_{0}^{(n)}$, and hence
$x^{(n)},y^{(n)}\in X\setminus U$, for otherwise they would lie in
different atoms of $\mathcal{Q}_{0}$, and hence project to different
points in $X_{0}$. 

Now pass to a subsequence so that $x^{(n)}\rightarrow x$ and $y^{(n)}\rightarrow y$.
We still have $x_{0}\neq y_{0}$ so $x\neq y$, and since $x^{(n)},y^{(n)}\in X\setminus U$
and $U$ is open, also $x,y\in X\setminus U$. On the other hand every
common shift $\widehat{x},\widehat{y}$ of $x,y$ are limits of corresponding
shifts $\widehat{x}^{(n)},\widehat{y}^{(n)}$ of $x^{(n)},y^{(n)}$,
and $\widehat{x}^{(n)},\widehat{y}^{(n)}$ still project to the same
point in $X_{n}$, so they agree on the central block $[-N_{n},N_{n}]$
except a $1/n$-fraction of the coordinates. This means that the same
is true for $\widehat{x},\widehat{y}$. What this argument shows is
that on every block of length $2N_{n}$, the points $x,y$ agree on
a $1-\frac{1}{n}$ fraction of the coordinates. Thus $\overline{d}(x,y)<1/n$
for every $n$, hence $\overline{d}(x,y)=0$.

To construct of the factor $X_{n}$ (the partition $\mathcal{Q}_{n}$),
for a large $L_{n}\in\mathbb{N}$ one wants to encode enough information
about the $(\mathcal{P},L_{n})$-names so as to reveal a $1-1/n$
fraction of the coordinates, but omit enough information that the
entropy remains below $h_{\mu}(T)$. To achive this, one partitions
the typical $(\mathcal{P},L_{n})$-names into families of names that
are close in the hamming distance, and combinatorial estimates show
that one can do so in a way that replacing $(\mathcal{P},L_{n})$-names
by the name of the corresponding family gives the desired result.

\subsection{The infinite-entropy case}

The reason this proof fails when $h_{\mu}(T)=\infty$ is that we cannot
realize it with a finite generator (finite alphabet); and if we use
an infinite alphabet we lose compactness, and cannot pass to the limit
$x,y$ of $x^{(n)},y^{(n)}$. 

Our observation is that the entire argument can be carried out relative
to a factor $Z$ of $X$, provided that the relative entropy is positive
and finite. Indeed, a measure-theoretic factor $Z$ of finite relative
entropy can be obtained using standard techniques, or in a more heavy-handed
manner by appealing to the Weak Pinsker property \cite{Austin2018}.
One can then realize the factor $X\rightarrow Z$ topologically with
$X=Z\times\{1,a\}^{\mathbb{Z}}$ and the factor being projection to
the first coordinate. From this point, the entire argument proceeds
as before. One should note that the construction of the extension
$X_{n}$ of $X_{n-1}$ in Ornstein and Weiss's original paper already
involves a relative argument, considering typical names in $X$ relative
to $X_{n-1}$, and this can be carried out just as well when $X_{n-1}$
contains the factor $Z$. We leave the details to the interested reader.

\bibliographystyle{plain}
\bibliography{bib}

\bigskip
\bigskip
\footnotesize
\noindent{}\texttt{Department of Mathematics, The Hebrew University of Jerusalem, Jerusalem, Israel\\ email: michael.hochman@mail.huji.ac.il}
\end{document}